\documentclass{amsart}
\usepackage{amsthm,amstext,amsmath,amscd,amssymb,latexsym}
\usepackage{color}
\usepackage{amsfonts,enumerate,array}
\usepackage[colorlinks=true]{hyperref}
\usepackage{todonotes}
	\usepackage{verbatim}
\usepackage{float}	
\usepackage{pgf,tikz}
\usepackage[toc,page]{appendix}
\usepackage[all,cmtip]{xy}
\usepackage{stmaryrd}
\usepackage{xparse}
\usepackage{listings}

\newcommand{\CC}{{\mathbb C}}
\newcommand{\PP}{{\mathbb P}}

\newcommand{\pp}{{\mathbb P}}

\newcommand{\Bl}{{\rm{Bl}}}
\newcommand{\calZ}{{\mathcal{Z}}}

\NewDocumentCommand{\dslash}{s}{%
  \IfBooleanTF{#1}
    {\big/\mkern-7mu\big/}
    {/\mkern-6mu/}%
}

\DeclareMathOperator{\rank}{rank}

\DeclareMathOperator{\Pic}{Pic}

\DeclareMathOperator{\Cox}{Cox}

\def\Bl{\operatorname{Bl}}

\newcommand{\paper}{: \begin{it}}
\newcommand{\jour }{, \end{it}}

\newtheorem{theorem}{Theorem}[section]
\newtheorem{lemma}[theorem]{Lemma}
\newtheorem{proposition}[theorem]{Proposition}
\newtheorem{corollary}[theorem]{Corollary}

\newtheorem{question}[theorem]{Question}
\theoremstyle{definition}
\newtheorem{definition}[theorem]{Definition}
\newtheorem{notation}[theorem]{Notation}

\theoremstyle{remark}
\newtheorem{remark}[theorem]{Remark}

\numberwithin{equation}{section}

\title[Newton-Okounkov bodies $\&$ Toric Degenerations of Mori dream spaces]{Newton-Okounkov bodies and Toric Degenerations of Mori dream spaces via Tropical compactifications}
\author{Elisa Postinghel}
\address{ Department of Mathematical Sciences, Loughborough University, Leicestershire LE11 3TU, United Kingdom}
\email{e.postinghel@lboro.ac.uk}
\author{Stefano Urbinati}
\address{Dipartimento di Matematica, Politecnico di Milano, via Bonardi 9, Milano, Italy}
\email{urbinati.st@gmail.com}

\subjclass{}

\date{}
\thanks{The second author was partially supported by the European Commission, Seventh Framework Programme, Grant Agreement n$^{\circ}$ 600376.}

\subjclass[2010]{Primary: 14C20. Secondary: 14E25, 14M25, 14E30}

\begin{document}
\begin{abstract}
Given a smooth Mori dream space $X$ we construct a model dominating all the small $\mathbb{Q}$-factorial modifications via tropicalization. 
This construction allows us to recover a Minkowski basis for the Newton-Okounkov bodies of divisors on  $X$ and hence the movable cone of $X$. 
The existence of such basis allows us to prove the polyhedrality of the global Newton-Okounkov body and the existence of toric degenerations.
\end{abstract}

\maketitle

\section{Introduction}

In this work we make a connection between the theory of Newton-Okounkov bodies and tropical geometry, sharing the aim of the recent preprints \cite{KU,KM}.
The results of the present parer yield a different point of view on how tropical geometry can be extremely helpful in birational geometry. 

Our principal aim is to  study Mori dream spaces (MDS) via tropicalization. Via this connection we will describe a simple and computable way of reconstructing the movable cone of such varieties. 
The results obtained complete the picture introduced in \cite{LS,PSU}, where Minkowski bases for Newton-Okounkov bodies were given respectively for surfaces and for toric varieties, with respect to certain admissible flags.

Mori dream spaces are special projective varieties for which the Minimal Model Program is particularly simple, since they only admit a finte number of \emph{small $\mathbb{Q}$-factorial modifications} (SQM's). 
The key property of these varieties is that they always admit a particularly nice embedding into toric varieties, see \cite{HK}. Given a MDS $X$ and such an embedding $X \subseteq Z$ into a toric variety, let $T \subseteq Z$ be the maximal torus of the given toric variety. The main result of this paper, that is Theorem \ref{main}, can be summarized as follows:

\begin{theorem}\label{main thm 1 intro}
Let $X\subset Z$ be a MDS with the embedding of \cite{HK} into a toric variety $Z$. Then the tropicalization of the variety restricted to $T$, ${\rm Trop}(X|_T)$, induces a  model $h:\overline{X}\to X$, embedded in a toric variety $j:\overline{X}\subset\overline{Z}$, that dominates all the SQM's induced by the Minimal Model Program.  Moreover the fan of $\overline{Z}$ is supported on ${\rm Trop}(X|_T)$.
\end{theorem}

Via this construction we are able to prove several consequences.
First of all, the map $h$ is given as an embedded map into toric varieties and this allows us to obtain a Minkowski basis for the Newton-Okounkov bodies on $X$. In particular we can reconstruct the movable cone of $X$. The main result is the following, more details are given in Lemma \ref{flag} and Theorem \ref{base}.

\begin{theorem} \label{main thm 2 intro} 
In the notation of Theorem \ref{main thm 1 intro},  certain admissible flags on $\overline{Z}$ induce admissible flags on $X$, such that if $\{D^Z_i\}_{i\in I}$ is a set of generators of the nef cone of $\overline{Z}$, then  $\{D_i:=h_*j^*D^Z_i\}_{i\in I}$ is a Minkowski basis for $X$ with respect to induced flag. \end{theorem}

Moreover, another main result of this paper is that the Newton-Okounkov bodies of the Minkowski basis elements are rational and polyhedral. Moreover the value semigroup is finitely generated. This implies the following result:
\begin{theorem}\label{main thm 3 intro}
If $X$ is a smooth Mori dream space, the global Newton-Okounkov body of $X$ with respect to the flags obtained in Theorem \ref{main thm 2 intro} is a rational polyhedron. 

Moreover if $A$ is an ample line bundle on $X$, $(X,A)$ admits a flat embedded degeneration to a not necessarily normal toric variety whose normalisation is the toric variety defined the Newton-Okounkov body of $A$. 
\end{theorem}
Details are given in Corollary \ref{global} and Theorem \ref{degeneration}. The first statement of Theorem \ref{main thm 3 intro} answers affirmatively a question posed by Lazarsfeld and Musta{\c{t}}{\u{a}} in \cite{LM}. The second statement is based on work of Anderson \cite{dave}.


The article is organized as follows. In section \ref{GS} we recall the notations and basic properties of MDS's and Cox rings. We review the constructions via GIT of the embedding of a MDS into a toric variety and we deduce the connection with tropical geometry. In section \ref{section tropicalization} we construct the universal model of MDS's via tropical geometry. In section \ref{MB}, after recalling the constructions of Newton-Okounkov bodies and Minkowski bases, we prove our main result, i.e. we show the existence of a Minkowski basis for a MDS, we show how to construct it and we prove Theorem \ref{main thm 3 intro}. Finally, in Section \ref{examples}, we give some examples of the construction, explicitly computed via the free software  \emph{Gfan} \cite{gfan}.

\section{General setting} \label{GS}
In this section we collect several results from \cite{GM, HK, KSZ, tevelev} providing the main connections between tropicalization and Mori dream spaces.

\subsection{Newton-Okounkov bodies and toric degenerations}

Let $X$ be a smooth complex projective variety of dimension $n$ over an algebraically closed field $k$ and let $D$ be a Cartier  divisor on $X$. 
Okounkov's construction associates to $D$ a convex body
$$
\Delta_{Y_{\bullet}}(D) \ \subseteq \ \mathbb{R}^n,
$$
which we call the \emph{Newton-Okounkov body}.
It depends on the choice of an \emph{admissible flag} $Y_{\bullet}$, i.e. a chain 
$$Y_0 =  X  \supset  Y_0  \supset  \ldots  \supset  Y_{n-1}  \supset  Y_n ,$$
where $Y_i$ is a subvariety of codimension $i$ in $X$ which is smooth at the point $Y_n$.
Using the flag, one defines a rank $n$ \emph{valuation} $\nu=\nu_{Y_{\bullet}}$ which, in turn, defines a \emph{graded semigroup} $\Gamma_{Y_{\bullet}}\subseteq \mathbb{N}\times\mathbb{N}^n$.
The convex body $\Delta_{Y_{\bullet}}(D)$ is the intersection of $\{1\}\times \mathbb{R}^n$ with the closure of the convex hull of $\Gamma_{Y_{\bullet}}$ in 
$\mathbb{R}\times\mathbb{R}^n$. 

One can also define the notion of \emph{global Newton-Okounkov body} of $X$ which is the closed convex cone in $\mathbb{R}^n\times N^1(X)_\mathbb{R}$ whose fibers over any big divisor $D$ on $X$ coincides with the Newton-Okounkov body $\Delta_{Y_{\bullet}}(D)$ of such divisor.
We refer to \cite{KK,LM}  for details on this construction.  

Newton-Okounkov bodies are quite difficult to compute in general. They often are not polyhedral and, when polyhedral, they may be not rational. Even when the body is polyhedral, the semigroup $\Gamma_{Y_{\bullet}}$ need not be finitely generated.

For toric varieties Newton-Okounkov bodies turn out to be nice. In fact
Lazarsfeld and Musta\c{t}\u{a} \cite{LM} proved that if $X$ is a smooth \emph{toric variety}, $D$ is a $T$-invariant ample divisor on $X$ and the $Y_i$'s are $T$-invariant subvarieties, then $\Delta_{Y_{\bullet}}(D)$ is the lattice polytope associated with $D$ in the usual toric construction. Moreover the global Newton-Okounkov bodies are rational polyhedral.

It is natural to investigate whether this construction behaves well for special classes of varieties.  A consequence of the work contained in 
\cite{BCHM,HK} is that divisors on Mori dream spaces often have toric-like behaviour, so it makes sense to pose the following question.

\begin{question}[{\cite[Problem 7.1]{LM}}]\label{question global body}
Let $X$ be a smooth Mori dream space. Does there exist an admissible flag with respect to which the global Newton-Okounkov body of $X$ is rational polyhedral?
\end{question}

In \cite{dave} Anderson extended the connection between Okounkov bodies and toric varieties, by introducing a geometric criterion for $\Delta_{Y_\bullet}(A)$ to be a lattice polytope for $A$ ample and, in this situation, by  constructing an embedded \emph{toric degeneration} of $(X,A)$.

\begin{theorem}[{\cite[Theorem 5.8]{dave}}]\label{theorem anderson} Let $A$ to be an ample divisor on $X$ and assume the value semigroup associated to $A$ with respect to the valuation induced by a complete flag is finitely generated. Then $X$ admits a flat degeneration to a toric variety whose normalization is  $X_{\Delta_{Y_{\bullet}}(A)}$. 
\end{theorem}
Another natural question is the following.

\begin{question}\label{question degeneration}
For which varieties is it possible to find a flag such that the value semigroup of an ample divisor is finitely generated? \end{question}
Notice that the latter is a very strong condition, and it is much stronger than the finite generation of the divisorial ring.

We will give an affirmative answer to both Questions \ref{question global body} and \ref{question degeneration} for Mori dream spaces in Section \ref{MB}.

\subsection{Mori dream spaces}

Let $X$ be a $\mathbb{Q}$-factorial projective variety  of dimension $n$ over the complex numbers. Let $r$ be the rank of ${\rm Pic}_{\mathbb{Q}}(X)={\rm N}^1(X)$ and let $\{L_1, \ldots , L_r\}$ be a set of line bundles on $X$ whose numerical classes form a basis of ${\rm N}^1(X)$ and whose affine hull contains $\overline{{\rm NE}}^1(X)$.
\begin{definition} For any vector of integers $v=(n_1, \ldots, n_r)$, we set $L^v:=L_1^{\otimes n_1} \otimes \dots \otimes L_r^{\otimes n_r}$. The \emph{Cox ring} of $X$ is given by
$${\rm Cox}(X):= {\rm R}(X; L_1, \ldots, L_r)= \bigoplus_{v \in \mathbb{N}^r} H^0(X, L^v).$$

\end{definition}

\begin{definition}[\cite{HK}] Let $X$ be a $\mathbb{Q}$-factorial projective variety such that ${\rm Pic}_{\mathbb{Q}}(X)=N^1(X)$. Then we say that $X$ is a \emph{Mori Dream Space} (MDS) if and only if ${\rm Cox}(X)$ is a finitely generated $\mathbb{C}$-algebra.
\end{definition}

\begin{notation}
We denote by $N$ the dimension of ${\rm Cox}(X)$ over $\CC$ and we set $$m:=N-r,$$ where $r$ as above is the rank $N^1(X)$.
\end{notation}
\begin{notation}
A toric variety $Z$ is quasi-smooth, i.e. it is a finite abelian quotient of a smooth toric variety,  if the fan is given by rational simplicial cones.
\end{notation}
\begin{notation} We will denote by ${\rm Nef}(X)$ and ${\rm Mov}(X)$ the cones of \emph{nef} and \emph{movable} divisors of $X$ (see \cite{laz1} for precise definitions). Here we just recall that nef divisors are the ones intersecting positively all the effective curves and movable divisors are the ones with stable base locus of codimension at least $2$. 
\end{notation}

It is well known that that ${\rm Nef}(X)\subseteq {\rm Mov}(X)$; when $X$ is a MDS we know much more.

\begin{proposition}[{\cite[Proposition 2.11]{HK}}]\label{hu-keel embedding} Let $X$ be a Mori dream space. Then there is an embedding $X \subseteq Z$ into a quasi-smooth projective toric variety (with torus $T$) 
 such that:
\begin{enumerate}
\item the restriction ${\rm Pic}(Z)_{\mathbb{Q}} \to {\rm Pic }(X)_{\mathbb{Q}}$ is an isomorphism,
\item the previous isomorphism  induces an isomorphism $\overline{{\rm NE}}^1(Z) \to \overline{{\rm NE}}^1(X)$, 
\item every Mori chamber of $X$ is a union of finitely many Mori chambers of $Z$, 
\item for every rational contraction $\varphi: X \to X_0$ there is toric rational contraction $\varphi_T: Z \to Z_0$, that is regular at the generic point of $X$, such that $\varphi=\varphi_T|_X$.
\end{enumerate}
\end{proposition}

Combining the previous proposition and the definition of MDS's, Hu and Keel \cite{HK} proved that both the nef cone and the movable cone of a MDS are rational polyhedral. Moreover, the movable cone is reconstructed as a finite union of cones, the so called \emph{Mori chamber decomposition} as follows. 
Let us denote by $\mathcal{F}:=\{f_i:X \dasharrow X_i\}$ 
 the set of finitely many SQM's of $X$, then we have
$${\rm Mov}(X)= \bigcup_{f_i \in \mathcal{F}} f^*_i({\rm Nef (X_i)}).$$

\subsection{GIT and Chow quotients for toric varieties} 

The main references for this section are \cite{KSZ,tevelev}.
Let us consider a projective toric variety $W \subseteq \mathbb{P}^N$ with torus $T$. Let us assume that the $T$-action on $W$ extends to $\mathbb{P}^N$. 
Moreover let $G$ be a subtorus of its defining torus $T$. 
In \cite{KSZ} Kapranov, Sturmfels and Zelevisky constructed a quotient of such variety called the \emph{Chow quotient}, denoted by $W\dslash G$. 
\begin{proposition}[{\cite[Section 3]{KSZ}}]\label{GIT} 
The toric variety $(T\dslash G, W\dslash G)$ dominates each GIT quotient of $W$ with respect to $G$ and it is the minimal normal variety with such property. 
\end{proposition}

\begin{notation} \label{chowquotient} The variety $(T\dslash G, W\dslash G)$ will appear as $\overline{Z}$ in the rest of the paper.
\end{notation}
\medskip

Let us now consider $Y$ a connected closed subvariety of the algebraic torus $T$. Let $\overline{Y}$ be a compactification of $Y$ in a toric variety $\calZ$ of $T$.

\begin{definition}[{\cite[Definition 1.1]{tevelev}}]\label{tropical compactification} We say that $\overline{Y}$ is a \emph{tropical compactification} if the multiplication map $$\Phi: T \times \overline{Y} \to \calZ, \quad (t,y)  \mapsto ty$$ 
is faithfully flat and $\overline{Y}$ is proper.
\end{definition}
\begin{notation}
The following notation is commonly used: the pair $(\overline{Y},\mathcal{Z})$, with $\overline{Y}$ and $\mathcal{Z}$ as in Definition \ref{tropical compactification}, is a tropical compactification of $Y\subseteq T$. 
\end{notation}

\begin{theorem} [{\cite[Theorem 1.2]{tevelev}}] Any subvariety $Y$ of a torus has a tropical compactification $\overline{Y}$ such that the corresponding toric variety $\calZ$ is smooth. 
In this case the boundary $\overline{Y}\backslash Y$ is divisorial and has \emph{combinatorial normal crossings}: for any collection
$B_1, \ldots , B_r \in \overline{Y}\backslash Y$ of irreducible divisors, $\bigcap B_i$ has codimension $r$. 
\end{theorem}

\begin{proposition} [{\cite[Proposition 2.5]{tevelev}}] Assume that $(Y, \calZ)$ is a tropical compactification. Then the fan of $\calZ$ is supported on the tropicalization of $Y$, ${\rm Trop}(Y)$.

\end{proposition}

\begin{remark}
Note that the variety $\overline{Z}$ in Notation \ref{chowquotient} is projective but might be singular, whereas $\calZ$ in Definition \ref{tropical compactification} is never projective, since the associated fan is never complete, but it is always smooth.

\end{remark}

The key observation of Tevelev in \cite[Section 5]{tevelev} is that the tropical compactification is a refinement of the Chow compactification, i.e. every compactification induced by GIT is dominated by the tropical compactification. Tevelev also observes  that the previous statement might fail if the torus sits trivially or in special position in the variety, but we will see that with a good choice of an auxiliary compactification of the torus, i.e. the one given by \cite{HK}, the statement is always true.

We are now ready to introduce the key object linking tropical geometry and MDS's.
\begin{notation}\label{notation subfan} Let $X$ be a normal algebraic variety and $W$ a normal algebraic toric variety with $X\subseteq W$.
Let $\Sigma_W$ be the fan defining the toric variety $W$. We will consider the following subfan of $\Sigma_W$
$$\Sigma_{W,X}:= \{\sigma \in \Sigma| {\rm V}(\sigma)\cap X \ne \emptyset \}.$$
\end{notation}

This subfan has already been considered in closely related works (cf. \cite{GM,tevelev}). In the next section we will construct a birational model of the toric variety $W$ viewed as the toric variety associated to a subfan of the Chow quotient. The idea was first introduced  by Lafforgue in \cite{laf} in the case of Grassmannians.

\section{Universal model via tropicalization}\label{section tropicalization}

Let $X$ be a Mori dream space and let $X\subseteq Z$ be an embedding of $X$  into a toric variety $Z$, as  in Proposition \ref{hu-keel embedding}. 
Via  ambient toric modifications 
it is possible to give a sequence of blow-ups in order to obtain an inclusion $\overline{X}\subseteq \overline{Z}$ such that for every small $\mathbb{Q}$-factorial modification of $X$, $f_i:X \dasharrow X_i$, we have the commutative diagram of Figure \ref{diagram},
where the $Z_i$ corresponds to the toric GIT quotient for the given SQM. Note that the diagram exists because MDS's admit finitely many SQM's.
A similar construction is studied in \cite{hausen}.

\begin{figure}[h!]
\begin{center}
\begin{tikzpicture}[line cap=round,line join=round,x=.7cm,y=.7cm]

\node (A) at (0,0) {$X$};

\node (B) at (2,0) {$X_i$};

\node (C) at (-1,0) {$Z$};

\node (D) at (3,0) {$Z_i$};

\node (E) at (1, 2) {$\overline{X}$};

\node (F) at (1,3) {$\overline{Z}$};

\node at (-.5,0) {$\supseteq$};
\node at (2.5,0) {$\subseteq$};
\node at (1, 2.5) {\rotatebox{90}{$\subseteq$}};

\draw[->] (E) -- (A); 
\draw[->] (E) -- (B); 
\draw[->] (F) [out=210, in=85, ->] to  (C); 
\draw[->] (F) [out=-30, in=95, ->] to (D); 
\draw[dashed,->] (A) -- (B);

\node at (1, -.3) {$_{f_i}$};

\node at (-.6,2) {$_{g}$};
\node at (2.8, 2) {$_{g_i}$};

\node at (.2,1) {$_{h}$};
\node at (1.8, 1) {$_{h_i}$};
\end{tikzpicture}
\end{center}
\caption{Universal commutative diagram}
\label{diagram}
\end{figure}
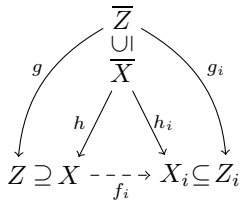

\begin{theorem} \label{main}
Let $X$ be a smooth Mori dream space with $\rank(N^1(X))=r$ and let $X \subseteq Z$ be the embedding of Proposition \ref{hu-keel embedding}, where $Z$ is a toric variety with maximal torus $T$. Set $X_T:=X|_T$, the restriction of $X$ to the maximal torus of $Z$. Then the following holds.
\begin{enumerate}
\item There exists a quasi-smooth toric variety $\overline{Z}$ and an embedded variety $\overline{X}\hookrightarrow \overline{Z}$ resolving all the small $\mathbb{Q}$-factorial modifications of $X$ induced by the Minimal Model Program.

\item Let $\Sigma_{\overline{Z}}$ be the fan associated to $\overline{Z}$ and let and $\Sigma_{\overline{Z}, \overline{X}}$  be the subfan of $\Sigma_{\overline{Z}}$ whose $T$-invariant orbits intersect $\overline{X}$, as in Notation \ref{notation subfan}.
Then $\Sigma_{\overline{Z}, \overline{X}}$ is supported on ${\rm Trop}(X_T)$.
\end{enumerate}
\end{theorem}

\begin{proof} Since MDS's admit only a finite number of SQM's, it is always possible to construct a variety resolving all such morphisms. It is well known that the Mori chambers of $X$ are refined by the Mori chambers of $Z$, cfr. Proposition \ref{hu-keel embedding} (3). This implies  that all of the SQM's are induced by a sequence of SQM's at the level of the ambient toric variety. 
Therefore  $\overline{Z}$ is obtained as a finite sequence of star subdivisions corresponding to the resolution of all the SQM's. This  proves $(1)$. 

By the previous argument we know that there exists a normal toric variety $\overline{Z}$ that  dominates all possible GIT quotients associated to the ${\rm Cox}$ ring of $X$. 
Then $\overline{Z}$ must  be the Chow quotient of Proposition \ref{GIT}.

Let $\mathbb{T}_{\overline{Z}, \overline{X}}$ denote the toric variety associated to $\Sigma_{\overline{Z}, \overline{X}}$.
We claim that  $(\overline{X},\mathbb{T}_{\overline{Z}, \overline{X}})$ 
is a tropical compactification of $X_T \subseteq T$. Since the inclusion is obviously proper, as it is constructed from an embedding via a sequence of blow-ups, the pair is a tropical compactification if and only if the boundary divisors have combinatorial normal crossings (cfr. \cite[Proof of Theorem 1.2]{tevelev}).

Note that proposition \ref{hu-keel embedding} implies that the intersection has the correct dimension for divisors.  That these restrictions, when non-empty, are  with combinatorial normal crossings it is not true a priori.
 However the desired property  can be achieved via a sequence of star subdivisions of the defining fan of the ambient toric variety; these star subdivisions giving rise to $\overline{Z}$. This concludes the proof of $(2)$. 
\end{proof}

\begin{remark}
Note that in the proof of Theorem \ref{main} we implicitly showed that the tropicalization itself detects what are the star subdivisions to be performed in order to solve such issue: these correspond to the blow-up of the loci with the ``incorrect'' intersection dimension. In particular we confirmed that the Chow quotient is dominated by the tropicalization as stated in \cite[Section 5]{tevelev}.
\end{remark}

\begin{corollary}\label{tildeZ} 
There exists a smooth toric resolution $\widetilde{Z}$ of $\overline{Z}$ such that $\overline{X}$ embeds in $\widetilde{Z}$ and such that (1) and (2) of Theorem \ref{main} hold for the inclusion $\overline{X}\hookrightarrow\widetilde{Z}$.
\end{corollary}
\begin{proof}
Since $\mathbb{T}_{\overline{Z},\overline{X}}$ is smooth, it is possible to choose a smooth refinement of the fan $\Sigma_{\overline{Z}}$ of $\overline{Z}$, $\pi:\Sigma_{\widetilde{Z}}\to\Sigma_{\overline{Z}}$ such that the resulting morphism $\widetilde{X}\to\overline{X}$ is an isomorphism and the induced morphism of fans $\Sigma_{\widetilde{Z},\overline{X}}\to\Sigma_{\overline{Z},\overline{X}}$ is the identity. 
\end{proof}

Using this construction we can recover the movable cone of $X$. 
\begin{proposition} \label{cones}
With the above notation, we have that $h_*{\rm Nef}(\overline{X})= {\rm Mov}(X)$.
\end{proposition}
\begin{proof}
First of all we have that for every  SQM of $X$, $f_i:X \dasharrow X_i$, then $h_i^*{\rm Nef}(X_i)\subseteq {\rm Nef}(\overline{X})$. In particular, since $X$ is a MDS and $\overline{X}$ dominates all the SMQ's, this implies that ${\rm Mov}(X)\subseteq h_*{\rm Nef}(\overline{X})$.

For the  converse inclusion, recall that  the push-forward of the movable cone via any birational morphism is contained in the movable cone of the image, see for instance \cite[Corollary 3.12]{FL}. Therefore we have the following inclusions of cones: $h_*{\rm Nef}(\overline{X})\subseteq h_*{\rm Mov}(\overline{X})\subseteq {\rm Mov}({X})$.
\end{proof}

\section{Movable cones of Mori Dream Spaces} \label{MB}

The idea behind the notion of Newton-Okounkov bodies is that of associating polyhedral objects to line bundles on a given variety. It is then natural to try to understand the correlation between properties of polytopes and their algebro-geometric counterpart.

Let us recall the definition of Minkowski basis.

\begin{definition}[\cite{PSU}]
	Let $X$ be a smooth projective variety of dimension $n$ and 
	$Y_{\bullet} : X = Y_{0} \supseteq Y_{1} \supseteq \dots \supseteq 
	Y_{n-1} \supseteq Y_{n} = \{pt\}$ an admissible flag on $X$.
	A collection $\left\{D_1,\dots,D_r\right\}$ of pseudo-effective divisors 
	on $X$ is called \emph{Minkowski basis} of $X$ with respect to $Y_\bullet$ if
	\begin{enumerate}
		\item
			for any pseudo-effective divisor $D$ on $X$ there exist non-negative numbers 
			$\left\{a_1,\dots,a_r\right\}$ and a translation $\phi:\mathbb{R}^n\to \mathbb{R}^n$ such that 
			$$
				\phi(\Delta_{Y_\bullet}(D))=\sum a_i\Delta_{Y_\bullet}(D_i),
			$$
			\item
				the Newton-Okounkov bodies $\Delta_{Y_\bullet}(D_i)$ are indecomposable, 
				i.e., if $\Delta_{Y_\bullet}(D_i) = P_1 + P_2$ for convex bodies $P_1,P_2$, 
				then $P_j=k_j\cdot\Delta_{Y_\bullet}(D_i)$ for non-negative numbers
				$k_1,k_2$ with $k_1+k_2=1$.
	\end{enumerate} 
\end{definition}

The following is the key result linking tropical geometry and Newton-Okounkov bodies.

\begin{lemma} \label{flag}
Let $Y^Z_{\bullet} : \overline{Z} = Y^Z_{0} \supset Y^Z_{1} \supset \dots \supset 
	Y^Z_{m-1} \supset Y^Z_{m} $ be an admissible flag of torus-invariant subvarieties  of $\overline{Z}$ of dimension $\dim(Y^Z_{i})=m-i$, for $i=0,\dots,m$. Then $Y^Z_{\bullet}$ induces a flag on $\overline{X}$ via truncation to the first $n$ elements, restriction of such element to $\overline{X}$ and the choice of an irreducible element of the intersection: 
$$Y_{\bullet} : \overline{X}=Y_0\supset Y_1\supset \cdots\supset Y_n, \ \textrm{ where } Y_i\subseteq Y^Z_{m-n+i}|_{\overline{X}}$$
 of dimension $\dim(Y_i)=n-i$, for $i=0, \ldots, n$, is an admissible flag on $\overline{X}$. 
\end{lemma}

\begin{proof}
By Theorem \ref{main}, $\overline{X}\subseteq \mathbb{T}_{\overline{Z},\overline{X}}$ 
is a tropical compactification, hence the intersection with $T$-invariant divisors is combinatorially normal crossings, i.e. the intersection with the torus-invariant flag elements has always the correct codimension. If the flag obtained in this way is not given by irreducible elements, we can derive an admissible one by choosing irreducible components of the restrictions $Y^Z_{m-n+i}|_{\overline{X}}$. This proves the statement.
\end{proof}

\begin{remark}
Notice that if $\overline{Z}$ is not smooth, an admissible flag might not exist. In this case, let us consider the resolution $\widetilde{Z}\to\overline{Z}$ constructed in Corollary \ref{tildeZ}. Since we have the identity $\Sigma_{\widetilde{Z},\overline{X}}\to\Sigma_{\overline{Z},\overline{X}}$, it is always possible, modulo reordering the terms, to obtain the desired flag starting from an admissible flag on $\widetilde{X}$.
\end{remark}

\begin{remark} The flags constructed in Lemma \ref{flag} are supported on $\overline{X}$. Those are infinitesimal flags for $X$, and all the properties of Newton-Okounkov bodies are preserved.
\end{remark}

\begin{theorem} \label{base}
In the notation of Section \ref{section tropicalization}, set $j: \overline{X} \hookrightarrow \overline{Z}$. Let $\{D^Z_i\}_{i\in I}$ be a Minkowski basis for the nef cone of $\overline{Z}$. Then $$ \mathcal{D}:=\{D_i:=h_*j^*D^Z_i\}_{i\in I}$$ is a Minkowski basis for $X$ with respect to the flag constructed in Lemma \ref{flag}.
\end{theorem} 

\begin{proof}
Note that the inclusion $j: \overline{X} \hookrightarrow \overline{Z}$ allows us to reconstruct the Nef cone of $\overline{X}$, 
in fact \cite[Proposition 5.1]{GM} proves that $j^*\rm Mov(\overline{Z})={\rm Nef}(\overline{X})$. Then the movable cone on $X$ can be reconstructed via push-forward and Proposition \ref{cones}.

We claim that any valuation vector induced by the flag constructed in Lemma \ref{flag} is obtained by the truncation of a valuation vector of the ambient toric variety. In order to prove the claim, 
let us denote by $D^Z$ a divisor on $\overline{Z}$ and by $D$ its restriction to $\overline{X}$. Let $\pi$ be the projection map $\pi: \mathbb{R}^m \to \mathbb{R}^n$, obtained via the truncation. We first need to prove that: 
$$\pi(\Delta_{Y^Z_{\bullet}}(D^Z))= \Delta_{Y_{\bullet}}(D).$$

Since the valuation vectors inducing $\Delta_{Y^Z_{\bullet}}(D^Z)$ are given by $T$-invariant sections, it is clear that $\pi(\Delta_{Y^Z_{\bullet}}(D^Z))\subseteq \Delta_{Y_{\bullet}}(D)$. 

In the general case it is quite difficult to prove the reverse inclusion, sometimes it might even be false. In this case, since $X$ is a MDS, the movable cone is covered by pull-back's of nef cones of small modifications. 
For every movable divisor $D$ on $X$, there exists an SQM $f_i: X \dashrightarrow X_i$ such that $D_i:= f_*D$ is nef. We can then prove the reverse inclusion via the diagram \eqref{diagram} using the following property and the fact that every nef divisor is a limit of ample divisors.
For any ample divisor $A$ on $\overline{Z}$ we have that 
$${\rm vol}_{\overline{Z}|_{\overline{X}}} (A) = {\rm vol}_{\overline{X}}(A|_{\overline{X}}) = A^n \cdot \overline{X},$$
i.e. the restricted volume coincides with the volume of the induced linear series via restriction, see \cite{ELMNP} for precise definitions and properties.  

Now, every movable divisor is a limit of ample divisors in some given small birational model.
Hence we can conclude that the two Newton-Okounkov bodies coincide.

Finally, the additivity of bodies of  $X$ is simply induced by the additivity of the corresponding bodies of $\overline{Z}$ (that holds for the results contained in \cite{PSU}), since this property is preserved under projection.

\end{proof}

\begin{remark}
Lemma \ref{flag} and Theorem \ref{base} give a direct connection between the Minkowski basis for the MDS $X$ and the toric variety where it embeds to. The explicit description also allows us to study restrictions of sections from the toric variety to the MDS.
\end{remark}

\subsection{Application 1}

As an immediate consequence of Theorem \ref{base} we can then state the following result, that answers positively the Question \ref{question global body}

\begin{corollary} \label{global} Let $X$ be a Mori dream space. The global Newton-Okounkov body if $X$ with respect to an admissible flag obtained as in Theorem \ref{base} is a rational polyhedral polytope in $\mathbb{R}^{n+r}$, where $n=\dim(X)$ and $r= \rank(N^1(X))$. 
\end{corollary}

\begin{remark}
Work  done recently led to partial answers to the question posed by Lazarsfeld and Musta{\c{t}}{\u{a}}, Question \ref{question global body}.
In \cite{global1} the author gave an answer to the question for surfaces and, in higher dimension, for Mori dream spaces with certain particular 
flags provided that such flags exist. In \cite{global2} the authors gave a criterion for the polyhedrality of the global Newton-Okounkov body depending on the existence of a Minkowski basis and they proved polyhedrality for certain homogeneous threefolds. 
\end{remark}

\subsection{Application 2}

Via our construction we are able to extend Theorem \ref{theorem anderson} to the case of arbitrary ample divisors on smooth Mori dream spaces. 

\begin{theorem} \label{degeneration}
Let $X$ be a smooth Mori dream space and let $\mathcal{D}$ be the Minkowski basis for $X$ constructed in Theorem \ref{base}. For any ample divisor $A$ on $X$, let $\Delta_{Y_{\bullet}}(A)$ be the Newton-Okounkov body constructed as a Minkowski sum of elements of $\mathcal{D}$. Then $X$ admits a flat degeneration to the toric variety $X_{\Delta_{Y_{\bullet}}(A)}$.
\end{theorem}

\begin{proof}
The main condition for the existence of the flat degeneration is the finite generation of the value semigroup. This condition yields that the extremal points of the Newton-Okounkov body shall correspond to the valuation of some section (i.e. the point is not a limit but an actual valuation vector). 

By the construction, the restriction of torus-invariant sections to $\overline{X}$ will give the exact valuation vector. Even more, the inclusion map will ensure that the section corresponds to a section of  chosen ample line bundle $A$.
\end{proof}

\section{Examples} \label{examples}
\subsection{Example 1}

Let  $X=\Bl_4(\PP^2)$ the del Pezzo surface of degree $5$, namely  the blow-up of $\PP^2$ in $4$ 
points in linearly general position. This is isomorphic to 
$\overline{M}_{0,5}$,. Let $H$ denote the pull-back of the class of a line of $\PP^2$ and let $E_1,\dots, E_4$ denote the classes of the exceptional divisors. 
We have 
$$\Pic(X)=\langle H,E_1,\dots,E_4\rangle.$$
After a change of basis, we can take $\{ H-E_3-E_4,E_1,\dots,E_4\}$ to be a set of generators of the Picard group.
Moreover 
$\Cox(X)$ is generated by the classes 
$$\{E_{ij}:=H-E_i-E_j: 1\le i< j \le 4\}\cup\{E_i:1\le i\le 4\}.$$
We now rewrite the generators of $\Cox$ in the new basis of $\Pic(X)$:
\begin{align*}
H-E_i-E_j&= (H-E_3-E_4)-E_i-E_j+E_3+E_4, \ \forall1\le i<j\le4\\
E_i&= E_i,\ \forall1\le i\le4.
\end{align*}
Using this, we can describe the map $$\deg:\Cox(X)\to \Pic(X)$$ with the following $5\times10$ matrix
\begin{center}
\begin{tabular}{lc}
  \ & \small{$\begin{array}{rrrrrrrrrr}
E_{12}&E_{13}&E_{14}&E_{23}&E_{24}&E_{34}&E_1&E_2&E_3&E_4\end{array}$} \\
 \small{$\begin{array}{l} E_{34}\\ 
E_1\\E_2\\E_3\\E_4\end{array}$}  & $\left(
\begin{array}{rrrrrrrrrr}
\ \ 1  &\ \ 1 \ &\ 1\ &\ 1\ &\ \ 1\ &\ \ 1\ &\ \ 0\ &\ \ 0\  &\ \ 0\ &\ \ 0\ \\
-1 & -1&  -1&0&0&0&1&0&0&0\\
 -1 &0&0&-1&-1&0&0&1&0&0\\
1&0&1&0&1&0&0&0&1&0\\
1&1&0&1&0&0&0&0&0&1\\
\end{array}\right).$ 
\end{tabular}
\end{center}

Let $A$ be the $5\times5$ matrix such that the above matrix is the concatenated matrix
$$
\left(
\begin{array}{c|c}
A& I_5
\end{array}\right).
$$
The matrix representing $$\ker(\deg)\to \Cox(X)$$ is defined by the Gale transform 
$$\left(\begin{array}{c}
I_5\\
\hline
-A
\end{array}\right)=
\left(
\begin{array}{ccccc}
1&0&0&0&0\\
0&1&0&0&0\\
0&0&1&0&0\\
0&0&0&1&0\\
0&0&0&0&1\\
-1&-1&-1&-1&-1\\
1&1&1&0&0\\
1&0&0&1&1\\
-1&0&-1&0&-1\\
-1&-1&0&-1&0
\end{array}
\right).
$$

The rows of the above matrix define the fan  $\mathbb{Z}^5$ of a blow-up of $\PP^5$.
Precisely, the first six rays are the rays of the fan of $\PP^5$ in the standard 
basis of $\mathbb{Z}^5$. Let $x_1,\dots, x_5,x_0$ be corresponding homogeneous 
coordinates.
The four last rays define  
exceptional divisors of the blow-up of $T$-invariant $2$-planes of $\PP^5$, with $T=(\mathbb{C}^\ast)^5$,  defined by
\begin{align*}
\Pi_1=&\{x_1=x_2=x_3=0\},\\
\Pi_2=&\{x_1=x_4=x_5=0\},\\
\Pi_3=&\{x_0=x_2=x_4=0\},\\
 \Pi_4=&\{x_0=x_3=x_5=0\}.
\end{align*}
Notice that $\{\Pi_i\cap\Pi_j:i\neq j\}$ is the set of the six coordinate points of $\PP^5$.

The map $i:X\to Z$ is the embedding described in Proposition \ref{hu-keel embedding}.

\begin{remark} The embedding might not be unique since it  depends on the linearization by an ample line bundle. However the tropical compactification will be the same. 
Levitt in \cite{Lev} computed a set of equations for such an embedding that do not respect intersection conditions and cannot be the image of the embedding constructed by \cite{HK}.  
\end{remark}

Let $P$ be the $2-$plane of $\pp^5$ described by the following equations:
$$x_2+x_4+x_0=0, \quad x_1+x_4+x_5=0, \quad x_2+x_3+x_4+x_5=0,$$
with the following intersections
\begin{align*}
p_1:=&P\cap \Pi_1=[0,0,0,1,-1,-1],\\
p_2:=&P\cap \Pi_2=[0,1,-1,0,0,-1],\\
p_3:=&P\cap \Pi_3=[-1,0,-1,0,1,0],\\
p_4:=&P\cap \Pi_4=[-1,-1,0,1,0,0].\\
\end{align*}
The image of $X$ inside $Z$ is the blow-up of $P$ at $p_1,p_2,p_3,p_4$.

\medskip 

Using \emph{Gfan}, we computed the tripicalization of $P\subset \PP^5$ obtaining the following description.

\begin{lstlisting} 
Q[x1,x2,x3,x4,x5]
{x2+x4+1,
x1+x4+x5,
x2+x3+x4+x5
}
LP algorithm being used: "cddgmp".
_application fan
_version 2.2
_type SymmetricFan

AMBIENT_DIM
5

DIM
2

LINEALITY_DIM
0

RAYS
-1 -1 -1  0 0  # 0 
-1 0 0 -1 -1   # 1 
-1 0 0 0 0     # 2
0 -1 0 0 0     # 3
0 0 -1 0 0     # 4
0 0 0 -1 0     # 5 
0 0 0 0 -1     # 6
1  0  1  0  1  # 7
1  1  0  1 0   # 8 
1  1  1  1  1  # 9

N_RAYS
10

F_VECTOR
1 10 15

SIMPLICIAL
1

PURE
1

CONES

 Dimension 0
{}  

 Dimension 1
{0}{1}{2}{3}{4}{5}{6}{7}{8}{9}

 Dimension 2
{0 2}{1 2}{0 3}{0 4}{1 5}{1 6}{3 6}{4 5}{2 9}{3 7}{5 7}{4 8}
{6 8}{7 9}{8 9}

\end{lstlisting}

Let $\mathcal{E}_i$ denote the exceptional divisor of $\Pi_i$ in $Z$ and let  $\mathcal{H}$ denote the hyperplane class in $Z$. Let us use the notation $\mathcal{E}_{ij}:=\mathcal{H}-\mathcal{E}_i-\mathcal{E}_j$ for the sake of simplicity.
The $1-$ and $2-$dimensional cones of the fan correspond to the orbits in $Z$ that intersect $X$. The rays correpond to the divisors $\mathcal{E}_{i}$ and $\mathcal{E}_{ij}$ The $2$-cone $\{0,2\}$ corresponds to $\mathcal{E}_1 . \mathcal{E}_{12}$, the class of a hyperplane on the exceptional divisor $\mathcal{E}_1$ of $\Pi_1$ in $Z$;
the $2-$cone $\{3,6\}$ corresponds to the intersection $\mathcal{E}_{13}.\mathcal{E}_{24}$, that is the strict transform on $Z$ of the $3-$plane spanned by the points $\Pi_i\cap\Pi_j$, with $(i,j)\in\{1,3\}\times\{2,4\}$
Whilst, for example, let us consider the cone $\{0,1\}$ that corresponds to the intersection $\mathcal{E}_1 . \mathcal{E}_2$, that is the strict transform in $Z$ of the point of $\PP^5$ of coordinates $[0,0,0,0,0,1]$: this orbit does not intersect $X$ and it does not appear in the tropicalization.

 Note that the rays exactly coincide with the ones appearing in the fan of $Z$. This is explained by the fact that a surface does not admit flips and hence we have a unique chamber of the Movable cone given by the Nef cone for $X$.

The combinatorial normal crossing property of the boundary allows us to induce a flag from the ambient toric variety $Z$ to the MDS $X$.

Let us consider, for example, the maximal cone in $\Sigma_{Z}$ generated by $\{ \mathcal{E}_{13},\mathcal{E}_{24}, \mathcal{E}_{14}, \mathcal{E}_{23},  \mathcal{E}_{3}\}$. This gives a $T$-invariant flag by complete intersection. The tropicalization selects all the subcones that induce a flag on $X$: in this case  one could choose the following $\{\mathcal{E}_{13}, \mathcal{E}_{24}\}$, $\{\mathcal{E}_{14}, \mathcal{E}_{23}\}$, $\{\mathcal{E}_{13}, \mathcal{E}_{3}\}$ or $\{\mathcal{E}_{23}, \mathcal{E}_{3}\}$.

It is then a simple computation to show that the Minkowski basis is obtained by projection to the corresponding selected components.

\subsubsection{A toric degeneration of the del Pezzo surface of degree $5$}
We consider $X$ embedded in ${\mathbb{P}^5}$ with the linear system of cubics through four points, given by the (ample) anticanonical divisor $A=3H-E_1-E_2-E_3-E_4$.

Let us consider on $Z$ the admissible flag 
$$Y^Z_\bullet: Z\supset\mathcal{E}_{13}\supset (\mathcal{E}_{13}\cap\mathcal{E}_{24})\supset\cdots$$
that induces the following admissible flag on $X$
$$Y_\bullet: X\supset E_{13}\supset \{(E_{13}\cap E_{24})\},$$
where $E_{ij}$ denotes the strict transform on $X$ of the line in $\mathbb{P}^2$ passing through two of the four blown-up points, using Lemma \ref{flag}.
Following Theorem \ref{base}, we find a Minkowski basis  for $X$ which is given by the following divisor:
$$\
\{H, H-E_i, H-E_i-E_j, 2H-E_1-E_2-E_3-E_4: i,j, i\ne j\}.
$$

One can compute the corresponding Newton-Okounkov bodies. For instance
the Newton-Okounkov body of $H$ is a triangle with vertices $(0,0),(0,1),(1,0)$ in $\mathbb{R}^2$, while the Newton-Okounkov body of $2H-E_1-E_2-E_3-E_4$ is a segment with vertices $(0,0),(1,1)$.
The ample divisor $A$ admits the Minkowski decomposition $A=H+(2H-E_1-E_2-E_3-E_4)$ and in particular  its Newton-Okounkov body is obtained as the Minkowski sum of the two polytopes of $H$ and $2H-E_1-E_2-E_3-E_4$, see Figure \ref{figure cubica}.
This gives rise to a toric degeneration of the del Pezzo surface $(X,A)$ to the singular toric surface $X_{\Delta_{Y_{\bullet}}(A)}$ obtained as the blow-up of $\mathbb{P}^2$ in two pairs of infinitely near points.

\begin{figure}[h]
\begin{center}\begin{picture}(0,30)(0,0)
\put(0,0){\line(0,1){20}}
\put(0,0){\line(1,0){20}}
\put(0,20){\line(1,-1){20}}
\put(-1.5,-1.5){\tiny{$\bullet$}}\put(-1.5,18.5){\tiny{$\bullet$}}
\put(18.5,-1.5){\tiny{$\bullet$}}
\end{picture} \quad\quad \quad\quad  +\quad\quad 
\begin{picture}(0,30)(0,0)
\put(0,0){\line(1,1){20}}
\put(-1.5,-1.5){\tiny{$\bullet$}}\put(18.5,18.5){\tiny{$\bullet$}}
\end{picture} \quad\quad\quad\quad
=
\quad\quad
\begin{picture}(0,30)(0,0)
\put(0,0){\line(0,1){20}}
\put(0,0){\line(1,0){20}}
\put(20,0){\line(1,1){20}}
\put(0,20){\line(1,1){20}}
\put(20,40){\line(1,-1){20}}
\put(-1.5,-1.5){\tiny{$\bullet$}}\put(-1.5,18.5){\tiny{$\bullet$}}
\put(18.5,-1.5){\tiny{$\bullet$}}\put(18.5,18.5){\tiny{$\bullet$}}\put(38.5,18.5){\tiny{$\bullet$}}\put(18.5,38.5){\tiny{$\bullet$}}
\end{picture}
\caption{Newton-Okounkov body of $3H-\sum_{i=1}^4E_i$}\label{figure cubica}
\end{center}
\end{figure}
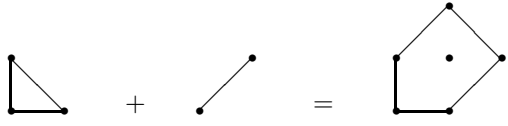


\subsection{Example 2}
Let now $X:=\Bl_{5}(\PP^3)$ be the blow-up of $\PP^3$ in $5$ points in linearly general position. Let us denote by $H$ the hyperplane class and by $E_1,\dots,E_5$ the exceptional divisors. 
The following divisors form a basis for the Picard group of $X$:
$$H-E_3-E_4-E_5, E_1,\dots,E_5.$$
Generators of the Cox ring of $X$ are:
$$\{H-E_{i_1}-E_{i_2}-E_{i_3}: 1\le i_1< i_2<i_3 \le 5\}
\cup\{E_i:1\le i\le 5\}.$$

Let $Y_\bullet$ be the flag on $X$ obtained starting from the flag $\{\mathcal{E}_{13}, \mathcal{E}_{13}\cap \mathcal{E}_{24},\dots\}$ on $Z$.

As in the previous example, the maps $\deg:\Cox(X)\to \Pic(X)$ and $\ker(\deg)\to \Cox(X)$
are defined, respectively,  by the following fundamental matrices:
$$
\left(
\begin{array}{c|c}
A& I_6
\end{array}\right):=
\left(
\begin{array}{rrrrrrrrr|rrrrrr}
1&1&1&1&1&1&1&1&1&1&0&0&0&0&0\\
-1&-1&-1&-1&-1&-1&0&0&0&0&1&0&0&0&0\\
-1&-1&-1&0&0&0&-1&-1&-1&0&0&1&0&0&0\\
0&1&1&0&0&1&0&0&1&0&0&0&1&0&0\\
1&0&1&0&1&0&0&1&0&0&0&0&0&1&0\\
1&1&0&1&0&0&1&0&0&0&0&0&0&0&1\\
\end{array}\right),
$$
and
$$\left(\begin{array}{c}
I_9\\
\hline
-A
\end{array}\right)=
\left(
\begin{array}{ccccccccc}
1&0&0&0&0&0&0&0&0\\
0&1&0&0&0&0&0&0&0\\
0&0&1&0&0&0&0&0&0\\
0&0&0&1&0&0&0&0&0\\
0&0&0&0&1&0&0&0&0\\
0&0&0&0&0&1&0&0&0\\
0&0&0&0&0&0&1&0&0\\
0&0&0&0&0&0&0&1&0\\
0&0&0&0&0&0&0&0&1\\
-1&-1&-1&-1&-1&-1&-1&-1&-1\\
1&1&1&1&1&1&0&0&0\\
1&1&1&0&0&0&1&1&1\\
0&-1&-1&0&0&-1&0&0&-1\\
-1&0&-1&0&-1&0&0&-1&0\\
-1&-1&0&-1&0&0&-1&0&0
\end{array}\right).
$$
The rows of the above matrix define a simplicial fan. The associated quasi-smooth projective variety is the variety  $Z$ that appears in Proposition \ref{hu-keel embedding}, that we can describe as follows. 
The first $9$ rows give rise to the fan of $\PP^9$, while the five last rows represent the
exceptional divisors of the blow-up of $\PP^9$ along the
torus invariant $3-$planes defined by
\begin{align*}
\Pi_1=&\{x_1=x_2=x_3=x_4=x_5=x_6=0\},\\
\Pi_2=&\{x_1=x_2=x_3=x_7=x_8=x_9=0\},\\
\Pi_3=&\{x_0=x_1=x_4=x_5=x_7=x_8=0\},\\
\Pi_4=&\{x_0=x_2=x_4=x_6=x_7=x_9=0\},\\
\Pi_5=&\{x_0=x_3=x_5=x_6=x_8=x_9=0\}.
\end{align*}
Therefore $Z$ is the blow-up of $\PP^9$ along  $\Pi_1,\dots,\Pi_5$.

\medskip 

We now describe the embedding $X\subset Z$ of Proposition \ref{hu-keel embedding}.
Consider the $3-$plane $P$ of $\PP^9$, that intersects each $\Pi_i$ in a point, defined by
\begin{align*}
x_1-x_2+x_3=0, \quad  x_1-x_7+x_8=0, \quad x_6-x_9+x_0=0\ \\
x_5-x_8+x_0=0, \quad x_2-x_7+x_9=0,\quad  x_1-x_4+x_5=0.
\end{align*}
The intersection points are given by
\begin{align*}
p_1:=&P\cap \Pi_1=[0,0,0,0,0,0,1,1,1,1],\\
p_2:=&P\cap \Pi_2=[0,0,0,1,1,1,0,0,0,-1],\\
p_3:=&P\cap \Pi_3=[0,1,1,0,0,-1,0,0,-1,0],\\
p_4:=&P\cap \Pi_4=[1,0,-1,0,-1,0,0,-1,0,0],\\
p_5:=&P\cap \Pi_5=[1,1,0,1,0,0,1,0,0,0].
\end{align*}
It is easy to check that $P\cap \Pi_i\subset\PP^9$ is a point, say $p_i$. 
The image of $X$ in $Z$ is the blow-up of $P$ in $p_1,\dots,p_5$.

\begin{remark}
Obviously the choice of a subspace $P$ satisfying the same properties is not unique. For instance 
the following set of equations describes another such a $3-$plane.
\begin{align*}
 x_1+x_7+x_8=0, \quad  x_6+x_9+x_0=0, \quad x_5+x_8+x_0=0,  & \\
 x_2+x_7+x_9=0, \quad x_4+x_5+x_7+x_8=0, \quad  x_3+x_6+x_8+x_0=0. &
\end{align*}

\end{remark}

\medskip

Using \emph{Gfan}, we computed the tripicalization of $P\subset \PP^9$ obtaining for both sets of equations the following 
description.

\begin{lstlisting} 
Q[x1,x2,x3,x4,x5,x6,x7,x8,x9]
{x1+x7+x8,
x6+x9+1,
x5+x8+1,
x2+x7+x9,
x4+x5+x7+x8,
x3+x6+x8+1
}
LP algorithm being used: "cddgmp".
_application fan
_version 2.2
_type SymmetricFan

AMBIENT_DIM
9

DIM
3

LINEALITY_DIM
0

RAYS
-1 -1 -1 -1 -1 -1 0 0 0    # 0 (*)
-1 -1 -1 0 0 0 -1 -1 -1    # 1 (*)
-1 -1 -1 0 0 0 0 0 0    # 2
-1 0 0 -1 -1 0 0 0 0    # 3
-1 0 0 0 0 0 -1 -1 0    # 4
-1 0 0 0 0 0 0 0 0    # 5 (*)
0 -1 0 -1 0 -1 0 0 0    # 6
0 -1 0 0 0 0 -1 0 -1    # 7
0 -1 0 0 0 0 0 0 0    # 8 (*)
0 0 -1 0 -1 -1 0 0 0    # 9
0 0 -1 0 0 0 0 -1 -1    # 10
0 0 -1 0 0 0 0 0 0    # 11 (*)
0 0 0 -1 0 0 0 0 0    # 12 (*)
0 0 0 0 -1 0 0 0 0    # 13 (*)
0 0 0 0 0 -1 0 0 0    # 14 (*)
0 0 0 0 0 0 -1 0 0    # 15 (*)
0 0 0 0 0 0 0 -1 0    # 16 (*)
0 0 0 0 0 0 0 0 -1    # 17 (*)
0 1 1 0 0 1 0 0 1    # 18 (*)
1 0 1 0 1 0 0 1 0    # 19 (*)
1 1 0 1 0 0 1 0 0    # 20 (*)
1 1 1 0 1 1 0 1 1    # 21
1 1 1 1 0 1 1 0 1    # 22
1 1 1 1 1 0 1 1 0    # 23
1 1 1 1 1 1 1 1 1    # 24 (*)

N_RAYS
25

F_VECTOR
1 25 105 105

SIMPLICIAL
1

PURE
1

\end{lstlisting}

The rays marked with $(\ast)$ are the rows of the matrix defining $Z$ (modulo a sign change). There are 10 more rays, that we are now going to interpret. 

Notice first that the span of each pair of $3-$planes $\langle\Pi_i,\Pi_j\rangle$ is a $6-$plane in $\PP^9$. Moreover the $3-$plane  $P$ intersects 
$\langle\Pi_i,\Pi_j\rangle$ in the line $\langle p_i,p_j\rangle$.
There are exactly $10$ of these $6-$planes in $\PP^9$ and, consequently, 10 of these lines in $P$.

Each of the remaining 10 rays above is the exceptional divisor of the blow-up of a (strict transform in $Z$ of the) $6$-plane $\langle\Pi_i,\Pi_j\rangle$. For instance the ray $\# 2$ (which is the sum of the rays $\# 5$, $\# 8$, $\# 11$ ) gives the exceptional divisors of the blow-up of the $6-$plane of $\PP^9$ of equations $x_1=x_2=x_3=0$, namely the one spanned by $\Pi_1$ and $\Pi_2$.

Now, in the notation of  Section \ref{section tropicalization}, the toric variety $\overline{Z}$ is the blow-up of $Z$ along all these 10 $6-$planes $\langle\Pi_i,\Pi_j\rangle$.
Moreover $\overline{X}\subset \overline{Z}$ is the blow-up of $X\subset Z$ along the 10 corresponding lines $\langle p_i,p_j\rangle$.
THerefore $\overline{X}\cong\overline{M}_{0,7}$, cfr. Section \ref{moduli section}.
The blow-up $\overline{X}$ 
resolves all the flips of $(-1)-$curves (given by the strict transforms of the  lines $\langle p_i,p_j\rangle$ in $X$).

The subfan $\Sigma_{\overline{Z},\overline{X}}$ is supported on the tropicalization of $P$  and it is described in the rest of the \emph{Gfan} scrip below. Notice that all rays and $2-$dimensional cones of  $\Sigma_{\overline{Z}}$ appear in $\Sigma_{\overline{Z},\overline{X}}$ because $\overline{X}$ intersects all $T$-invariant subvarieties of coidimension $1$ and $2$, just by a simple dimension count.

\medskip

\begin{lstlisting}
CONES

 Dimension 0
{}  

 Dimension 1
{0}{1}{2}{3}{4}{5}{6}{7}{8}{9}{10}{11}{12}{13}
{14}{15}{16}{17}{18}{19}{20}{21}{22}{23} {24}

 Dimension 2
{0 2}{1 2}{0 3}{0 5}{1 4}{1 5}{2 5}{3 5}{4 5}
{0 6}{0 8}{1 7}{1 8}{2 8}{0 9}{0 11}{1 10}{1 11}{2 11}
{0 12}{0 13}{0 14}{1 15}{1 16}{1 17}{3 12}{3 13}{3 17}
{4 14}{5 14}{4 15}{4 16}{5 17}{3 18}{4 18}{5 18}{6 8}{7 8}
{6 12}{6 14}{6 16}{7 13}{8 13}{7 15}{7 17}{8 16}{9 11}
{10 11}{10 12}{11 12}{9 13}{9 14}{9 15}{11 15}{10 16}
{10 17}{12 16}{12 17}{13 15}{13 17}{14 15}{14 16}{2 24}
{12 18}{13 18}{15 18}{16 18}{5 23}{5 24}{6 19}{7 19}{8 19}
{12 19}{14 19}{15 19}{17 19}{8 22}{8 24}{9 20}{10 20}{11 20}
{11 21}{13 20}{14 20}{16 20}{17 20}{11 24}{12 21}{15 21}
{13 22}{16 22}{14 23}{17 23}{18 21}{18 22}{18 24}{19 21}
{19 23}{19 24}{20 22}{20 23}{20 24}{21 24}{22 24}{23 24}


MAXIMAL_CONES  Dimension 3
{0 2 5} {1 2 5}{0 3 5}{1 4 5}
{0 2 8}{1 2 8}{0 2 11}{1 2 11}{0 3 12}
{0 3 13}{0 5 14}{1 4 15}{1 4 16}{1 5 17}{3 5 17}
{4 5 14}{3 5 18}{4 5 18}{0 6 8}{1 7 8}{0 6 12}
{0 6 14}{0 8 13}{1 7 15}{1 7 17}{1 8 16}{0 9 11}
{1 10 11}{0 11 12}{0 9 13}{0 9 14}{1 11 15}
{1 10 16}{1 10 17}{3 12 17}{3 13 17}{4 14 15}
{4 14 16}{2 5 24}{3 12 18}{3 13 18}{4 15 18}{4 16 18}
{6 8 16}{7 8 13}{6 12 16}{6 14 16}{7 13 15}{7 13 17}
{2 8 24}{10 11 12}{9 11 15}{10 12 16}{10 12 17}
{9 13 15}{9 14 15}{2 11 24}{12 16 18}{13 15 18}
{5 14 23}{5 17 23}{5 18 24}{6 8 19}{7 8 19}{6 12 19}
{6 14 19}{7 15 19}{7 17 19}{12 17 19}{14 15 19}
{8 13 22}{8 16 22}{9 11 20}{10 11 20}{9 13 20}
{9 14 20}{10 16 20}{10 17 20}{11 12 21}{11 15 21}
{13 17 20}{14 16 20}{12 18 21}{15 18 21}{13 18 22}
{16 18 22}{5 23 24}{8 19 24}{12 19 21}{15 19 21}
{14 19 23} {17 19 23} {8 22 24} {11 20 24} {11 21 24}
{13 20 22} {16 20 22} {14 20 23} {17 20 23} {18 21 24}
{18 22 24} {19 21 24} {19 23 24} {20 22 24} {20 23 24}
\end{lstlisting}


\end{document}